\documentclass[11pt]{amsart}
\usepackage{amssymb, adjustbox, enumerate, amsbsy, stmaryrd}
\usepackage{geometry}
\usepackage{todonotes}
\usepackage{amsfonts, amssymb, amscd}
\numberwithin{equation}{section}

\usepackage[symbol]{footmisc}

\usepackage{bm}
\usepackage{verbatim}
\usepackage{mathrsfs}
\usepackage{graphicx}
\usepackage{tikz-cd}
\usepackage{subcaption}
\usepackage{listings}
\usepackage{subfiles}
\usepackage[toc,page]{appendix}
\usepackage{mathtools}
\usepackage{comment}
\usepackage{enumerate}
\usepackage{enumitem}
\usepackage[all]{xy}

\usepackage{graphicx}
\graphicspath{{images/}}

\usepackage{appendix}
\usepackage{hyperref}
\hypersetup{
    colorlinks=true,
    citecolor=red,
    linkcolor=blue,
    filecolor=magenta,      
    urlcolor=red,
}
\def\ideal#1.{I_{#1}}
\def\ring#1.{\mathcal {O}_{#1}}
\def\fring#1.{\hat{\mathcal {O}}_{#1}}
\def\proj#1.{\mathbb P(#1)}
\def\pr #1.{\mathbb P^{#1}}
\def\af #1.{\mathbb A^{#1}}
\def\Hz #1.{\mathbb F_{#1}}
\def\Hbz #1.{\overline{\mathbb F}_{#1}}
\def\pic#1.{\operatorname {Pic}\,(#1)}
\def\pico#1.{\operatorname{Pic}^0(#1)}
\def\picg#1.{\operatorname {Pic}^G(#1)}
\def\ner#1.{NS (#1)}
\def\rdown#1.{\llcorner#1\lrcorner}
\def\rup#1.{\ulcorner#1\urcorner}
\def\cone#1.{\operatorname {NE}(#1)}

\def\ccone#1.{\overline{\operatorname {NE}}(#1)}
\def\coef#1.{\frac{(#1-1)}{#1}}
\def\vit#1.{D_{\langle #1 \rangle}}
\def\mm#1.{\overline {M}_{0,#1}}
\def\H1#1.{H^1(#1,{\ring #1.})}
\def\ac#1.{\overline {\mathbb F}_{#1}}

\def\adj#1.{\frac {#1-1}{#1}}
\def\spn#1.{\overline{#1}}
\def\ses#1.#2.#3.{0\to #1\to #2\to #3 \to 0}
\def\pek#1.#2.{\Cal P^{#1}(#2)}
\def\plk#1.#2.{\Cal P^{\leq #1}(#2)}
\def\ev#1.{\operatorname{ev_{#1}}}
\def\bminv#1.{(\nu_1,s_1;\nu_2,s_2;\dots ;\nu_{#1},s_{#1};\nu_{r+1})}
\def\zinv#1.{(\nu_1,s_1;\nu_2,s_2;\dots ;\nu_{#1},s_{#1};0)}
\def\iinv#1.{(\nu_1,s_1;\nu_2,s_2;\dots ;\nu_{#1},s_{#1};\infty)}
\def\map#1.#2.{#1 \longrightarrow #2}
\def\rmap#1.#2.{#1 \dasharrow #2}
\def\emb#1.#2.{#1 \hookrightarrow #2}

\def\Supp{\operatorname{Supp}}

\def\dim{\operatorname{dim}}

\def\mult{\operatorname{mult}}

\def\e{\Cal E}

\def\e1{E_1}
\def\e2{E_2}

\def\OO{\mathcal O}

\newcommand{\po}{\ar@{}[dr]|{\text{\pigpenfont R}}}
\newcommand{\pb}{\ar@{}[dr]|{\text{\pigpenfont J}}}

\newcommand\Q{{\mathbb{Q}}}
\newcommand\R{{\mathbb{R}}}
\newcommand{\exc}{\mathrm{Exc}}

\newtheorem{thm}{Theorem}[section]

\newtheorem{lem}[thm]{Lemma}

\theoremstyle{definition}
\newtheorem{defn}[thm]{Definition}

\theoremstyle{definition}
\newtheorem{rem}[thm]{Remark}

\theoremstyle{definition}

\begin{document}

\author{Christopher Hacon} \address{Department of Mathematics\\ University of Utah\\ 155 S 1400 E\\ Salt Lake City, Utah 84112} \email{hacon@math.utah.edu}
\thanks{Christopher Hacon was partially supported by the NSF research grants no: DMS-
1952522, DMS-1801851 and by a grant from the Simons Foundation; Award Number: 256202. The authors are grateful to J. Liu and O. Fujino for useful comments.}

\author{Lingyao Xie} \address{Department of Mathematics\\ University of Utah\\ 155 S 1400 E\\ Salt Lake City, Utah 84112} \email{lingyao@math.utah.edu}

\date{\today}

\title[ACC for glct for complex analytic spaces]{ACC for generalized log canonical thresholds for complex analytic spaces}

\begin{abstract}
    We show that generalized log canonical thresholds for complex analytic spaces satisfy the ACC and we characterize the accumulation points.
\end{abstract}
\maketitle

\section{Introduction}
Throughout this paper we work with pairs $(X,B)$ where $X$ is a normal complex analytic variety and $B$ is an effective $\R$-divisor such that $K_X+B$ is $\R$-Cartier. Understanding the singularities of such pairs plays a fundamental role in recent advances in the birational classification of algebraic varieties. One important measure of these singularities are the log canonical thresholds.
If $(X,B)$ is log canonical and $D$ is a non-zero effective $\R$-Cartier divisor, then the log canonical threshold is
\[{\rm lct}(X,B;D):={\rm sup}\{t|(X,B+tD)\ {\rm is\ log\ canonical}\}.\]
Understanding the behaviour of log canonical thresholds is essential in a variety of contexts such as, for example, the termination of flips, moduli problems, and K-stability (see, for example, \cite{Birkar07}, \cite{HMX18}, \cite{XZ21}).
Perhaps the most important result in this context is the solution, by Hacon-McKernan-Xu,  of Shokurov's ``ACC for LCT's conjecture" \cite{HMX14}
which we will now recall. 

 A set $I$ of non-negative real numbers satisfies the ascending chain condition or ACC (resp. the descending chain condition or DCC) if any non decreasing sequence $i_1\leq i_2\leq \ldots$ (resp. any non increasing sequence $i_1\geq i_2\geq \ldots$) is eventually constant.
Let $I,J$ be two DCC sets of nonnegative real numbers and $n$ a natural
number. We define \[\mathrm{LCT}_n(I,J):=\{{\rm lct}(X,B;D)|\dim X=n,\ {\rm coeff}(B)\in I,\ {\rm coeff}(D)\in J\}\]  to be the set of all log canonical thresholds of $n$-dimensional lc pairs $(X,B)$ with respect to divisors $D$ such that the coefficients of $B$ and $D$ belong to $I$ and $J$ respectively. When $X$ is quasi-projective, then by \cite[Theorem 1.1]{HMX14}, it follows that the set $\mathrm{LCT}_n(I,J)$ satisfies the ACC and that one can characterize the accumulation points under mild assumptions on the DCC sets $I,J$.
By \cite[Theorem 1.1]{HMX14}, it follows that if $I\subset [0,1]$, the only possible accumulation point of $I$ is $1$, and $I=I_+:=\{0\}\cup \{j=\sum _{k=1}^ri_k\in[0,1]|i_k\in I\}$, then the only accumulation points of $\mathrm{LCT}_n(I,\mathbb N)$ are $\mathrm{LCT}_{n-1}(I,\mathbb N)\setminus \{1\}$. 

Recently, generalized pairs have begun playing an increasingly more prominent role in birational geometry (see \cite{Birkar21} and references therein). It has become apparent that it is important to also study singularities in this context. The analog of the main results of \cite{HMX14}, for generalized pairs  where proven in  \cite{BZ16}, \cite{Liu18}, and have already found several important applications.

Naturally, it is also expected that the ACC for LCT's will play an important role in many of other contexts such as analytic varieties, foliated pairs, varieties in positive and mixed characteristics etc. (see eg. \cite{Fuj22b}, \cite{Chen22}, \cite{Sato21}).  
In view of recent progress in the minimal model program for analytic varieties (\cite{DHP22}, \cite{Fuj22a}), it is expected that the results of \cite{HMX14} should also hold for analytic varieties.
Fujino has in fact shown that Shokurov's ACC for LCT's conjecture holds for analytic varieties \cite{Fuj22b}. One interesting phenomenon that occurs in the analytic case (which does not happen in the algebraic case or for compact analytic varieties) is that if $\lambda = {\rm lct}(X,B;D)$, then it is possible that 
$(X,B+\lambda D)$ is klt i.e. there is no divisor $E$ over $X$ of log discrepancy $a(E;X,B+\lambda D)=0$ (see \cite[Example 1.3]{Fuj22b}). This is somewhat troubling as typically, many proofs by induction on the dimension involve studying the restriction of $K_X+B+\lambda D$ to an appropriate  divisor $E$ over $X$ of log discrepancy $a(E;X,B+\lambda D)=0$.

The purpose of this paper is to show that the results of \cite{HMX14} hold for generalized pairs on analytic varieties. In the process, we will show that log canonical thresholds $\lambda = {\rm lct}(X,B;D)$ are always computed by divisors of log discrepancy 0, for an auxiliary pair $(X',B'+\lambda D')$. We believe that results of this nature will find many applications in upcoming works on the minimal model program for K\"ahler varieties.

We will now give a more precise description of the main results of this paper.
Let $I,J$ be DCC sets and $n\in \mathbb N$, $f:X'\to X$ a proper biholomorphic map of analytic varieties, $(X, B + M)$, $M'$, $P'$, $P$, $D$ be as in Definition \ref{def: glct} so that
\begin{enumerate}
\item  $(X, B + M)$ is glc of dimension $n$,
\item $M' = \sum \mu_j M'_j$ where $M'_j$ are relatively nef Cartier divisors on $X'$, and  $\mu _j\in I$,
\item $P' = \sum \nu_kP'_k$ where $P'_k$ are relatively nef Cartier divisors on $X'$, and $\nu_k\in J$,
\item the coefficients of $B$ belong to $I$ and the coefficients of $D$ belong to $J$.
\end{enumerate}
Then $\mathrm{GLCT}_n(I,J)\subset \R$ is the set consisting of all the possible generalized log canonical thresholds $\mathrm{glct}(X,B+M;D+P)$ where $(X,B+M;D+P)$ are as above (see Definition \ref{def: glct}).

\begin{thm}\label{thm: glct(I,J) acc} The set 
$\mathrm{GLCT}_n(I,J)$ satisfies the ACC.
\end{thm}

Moreover, we  give a precise description of the accumulation points of  generalized log canonical thresholds as in \cite[Theorem 1.11]{HMX14} and \cite[Theorem 1.7]{Liu18}:

\begin{thm}\label{thm: accumulation pts come from lower dim}
If 1 is the only accumulation point of the DCC set $I\subset[0,1]$ and $1\in I=I_+$, then the accumulation points of $\mathrm{GLCT}_n(I):=\mathrm{GLCT}_n(I,\mathbb{N})$ belong to $\mathrm{GLCT}_{n-1}(I)$.
\end{thm}

\section{Preliminaries}

Let $X$ be a normal complex analytic space. A prime divisor $P$ on $X$ is an irreducible and reduced closed subvariety of codimension one. An $\R$-divisor (resp. $\Q$-divisor) $D$  on $X$ is a locally finite formal sum $D=\sum d_iD_i$ of distinct prime divisors $D_i$ with 
with coefficients $d_i\in \R$ (resp. $d_i\in \Q$). If for some point $x\in X$ there is a neighborhood $x\in U\subset X$ such that the restriction $D|_U$ of the $\R$-divisor (resp. $\Q$-divisor) $D$ is a finite $\R$-linear  (resp. $\Q$-linear) combination of Cartier divisors, then we say that $D$ is $\R$-Cartier (resp. $\Q$-Cartier) at $x\in X$. If $D$ is $\R$-Cartier (resp. $\Q$-Cartier) at every $x\in X$ then we say that $D$ is $\R$-Cartier (resp. $\Q$-Cartier).

\begin{defn}\label{def: g-pairs}
We say that $(X,B+M)$ is a generalized pair if there is a proper biholomorphic map $f:X'\to X$ and an $f$-nef $\R$-Cartier divisor $M'$ such that
\begin{enumerate}
    \item $X'$ and $X$ are normal,
    \item $M=f_*M'$ and $B\ge 0$,
    \item $K_X+B+M$ is $\R$-Cartier.
\end{enumerate}
We call $B$ the boundary part and $M$ the nef part of the generalized pair. We can always replace $X'$ by a higher model that factors through $f$, and $M'$ by its pullback. 

We can write 
$$
K_{X'}+B'+M'=f^*(K_X+B+M)
$$
and we say that the generalized pair $(X,B+M)$ is generalized log canonical (glc) at $x\in X$ if there is a neighborhood $x\in U\subset X$ such that $(X',B')|_{U}$ is sub-lc (\cite[Remark 3.2]{Fuj22a}) and is generalized kawamata log terminal (gklt) at $x\in X$ if there is a neighborhood $x\in U\subset X$ such that $(X',B')|_{U}$ is sub-klt (\cite[Remark 3.2]{Fuj22a}). We can also define $a(E,X,B+M):=a(E,X',B')$ for any divisor over $X$.
We say that $Z$ is a log canonical center (resp. log canonical place) of a $(X,B+M)$ glc pair if $Z$ is the image of an lc center of $(X',B')$ (resp. a log canonical place of $(X',B')$).  

We say that a glc pair $(X,B+M)$ is generalized divisorially log terminal (gdlt) if we can choose $f:X'\to X$ (in the definition) to be a log resolution of $(X,B)$ such that the log discrepancy is $a(E,X,B+M)>0$ for every $f$-exceptional divisor $E$.

\end{defn}

\begin{defn} A set $I\subset \R$ satisfies the ACC (resp. DCC) if any non-decreasing (resp. non-increasing) sequence $I_k\in I$ is eventually constant. We let $\partial I$ be the set of accumulation points of $I$ and $\bar I =I\cup \partial I$.
If $I\subset [0,+\infty),$ then \[I_+=\{0\}\cup \{\sum _{k=1}^li_k\in[0,1]~|~i_k\in I\}, \ and\]
\[D(I)=\{a \leq  1|a = \frac{m-1+f}
m ,\  m \in \mathbb  N^+,\  f \in I_+\}.\]
If $I\subset [0,1]$, then we let \[\Phi(I)=\{1 -\frac r
m |r \in I, m \in  \mathbb N^+\}.\]
\end{defn}
\begin{defn}\label{def: glct} (Generalized log canonical thresholds for complex analytic spaces). Let $(X, B+M)$ be
a generalized log canonical pair and let $D$ be an effective $\R$-Cartier $\R$-divisor on $X$ and $P=f_*P'$ where $P'$ is a nef divisor on $X'$. Let $c$ be the supremum of all real numbers such that $(X, B +M+ t(D+P))$ is generalized log canonical, then
$c$ is called the generalized log canonical threshold of $D+P$ with respect to $(X, B+M)$ and is  denoted
by ${\rm lct}(X, B+M; D+P)$. \end{defn}
\begin{lem} If $(X, B+M)$ and $D+P$ are as above, then $(X, B+M + c(D+P))$ is generalized log canonical.
\end{lem}
\begin{proof}
This follows directly from the definition.
\end{proof}
\begin{rem}\label{rem: lct in relative compact set create lc centers}
Note that the above definition differs from the one in \cite{Fuj22b}
as there does not necessarily exist a
non-kawamata log terminal center of   $(X, B+M + c(D+P))$.
The issue is that $X$ may not be compact. In this case, we may not have a log resolution, and the divisor $D$ may have infinitely many components see \cite[Example 1.3]{Fuj22b}. If however $X$ is (relatively) compact, then log resolutions exist, the two definitions agree, and we always have a log canonical center of $(X, B+M + c(D+P))$ see \cite[Remark 1.4.]{Fuj22b}.\end{rem}

The next theorem is the analogue of dlt-blowups of generalized pairs in the complex analytic setting:

\begin{thm}(Dlt-blowup)\label{thm: dlt-blowup for g-pairs}
Let $X$ be a normal complex variety and $(X,B+M)$ a generalized pair as in Defintion \ref{def: g-pairs}. %Assume that the coefficients of $B$ are $\leq 1$. 
Let $U$ be any relatively compact
Stein open subset of $X$ and let $V$ be any relatively compact open subset of $U$. Then we can
take a Stein compact subset $W$ of $U$ such that $\Gamma(W,\OO_X)$ is noetherian, $V\subset W$, and after
shrinking $X$ around $W$ suitably, we can construct a projective bi-meromorphic morphism $g:Y\to X$ from a normal complex variety $Y$ with the following properties:
\begin{enumerate}
\item $Y$ is $\Q$-factorial over $W$,
\item $a(E,X,B+M)\le0$ for every $g$-exceptional divisor $E$ on $Y$,
\item $(Y,B_Y^{\le1}+M_Y)$ is gdlt, where $K_Y+B_Y+M_Y=f^*(K_X+B+M)$.
\end{enumerate}

\end{thm}
\begin{proof}
We will freely shrink $X$ suitably without mentioning it explicitly.
By taking a resolution of singularities, we can assume that $f:X'\to X$ is a projective bi-meromorphic morphism such that $f^{-1}(U)$ is smooth and $\exc(f)\cup\Supp(f^{-1}B)$ is a simple normal crossing divisor on $f^{-1}(U)$. Let $E$ be any $f$-exceptional
divisor such that $f(E)\cap U\neq\emptyset$. Then, by enlarging $V$ suitably, we may assume that $f(E)\cap V\neq\emptyset$. By \cite[Lemma 2.16]{Fuj22a}, we can take a Stein compact subset $W$ of $U$ such that $\Gamma(W,\OO_X)$ is noetherian and that $V\subset W$.

Write $K_{X'}+B'+M'=f^*(K_X+B+M)$ as in Definition \ref{def: g-pairs} and let $B'=\sum a_iD_i$ be the irreducible decomposition. Now we define a boundary %\footnote{I think we need to add $\epsilon D_i$ for any $f$-exceptional divisor such that $0\leq   a_i< 1$.}
$$
\Delta=\sum_{0<a_i<1}a_iD_i+\sum_{a_i\ge1}D_i+\epsilon \sum E_i~,~1\gg\epsilon>0,
$$
where $E_i$ are all the $f$-exceptional divisors such that $a(E,X,B+M)>0$.
Then we have $K_{X'}+\Delta+M'=f^*(K_X+B+M)+F$ and we see that $-f_*F$ is effective. Let $A$ be a general ample (over $X$) $\Q$-divisor such that $(X',\Delta+A+M')$ is gdlt and $K_{X'}+\Delta+A+M'$ is nef over $W$. Notice that for any $t>0$, $tA+M'$ is $f$-ample. Therefore (over $W$) we can write $K_{X'}+\Delta+tA+M'\sim_{\Q,f}K_{X'}+\Delta^t$ for some klt pair $(X',\Delta^t)$. Then by \cite{Fuj22a} we can run a $(K_{X'}+\Delta+M')$-MMP with scaling of $A$ over $X$ over $W$.

Let $(X_0,\Delta_0+M'):=(X',\Delta+M')$, $F_0:=F$, $M_0:=M'$ and $A_0:=A$. Then we obtain a sequence of divisorial contractions and flips:
$$
(X_0,\Delta_0+M_0)\stackrel{\phi_0}{\dashrightarrow}(X_1,\Delta_1+M_1)\stackrel{\phi_1}{\dashrightarrow}\cdots\stackrel{\phi_{i-1}}{\dashrightarrow}(X_i,\Delta_i+M_i)\stackrel{\phi_i}{\dashrightarrow}
$$
where $\Delta_i,M_i,F_i,A_i$ are the corresponding birational transforms. We also have the scaling numbers
$$
1\ge\lambda_0\ge\lambda_1\ge\cdots\ge\lambda_i\ge\cdots\ge0
$$
such that $K_{X_i}+\Delta_i+M_i+\lambda_iA_i$ is nef over $W$. Then by \cite[Lemma 13.7]{Fuj22a} we know that $K_{X_k}+\Delta_k+M_k\in\overline{\mathrm{Mov}}(X_k/X;W)$ for some $k\ge0$. Thus by the negativity lemma (cf. \cite[Lemma 4.9]{Fuj22a}) we have $-F_k\ge0$ over $W$. Hence $-F_k$ is effective over some open neighborhood of $W$. Let $Y:=X_k$, $g:X_k\to X$, $M_Y=M_k$ and $K_Y+B_Y+M_Y=g^*(K_X+B+M)$. Then $(Y,B_Y+M_Y)$ satisfies (1-3) above.
\end{proof}

\section{Proof of the main theorems}

\begin{lem}\label{lem: glct make num trivial g-pair}
We fix a positive integer $n$ and a set $1\in I\subset[0,\infty]$. Assume that $f:X'\to X$ is a projective morphism, and we have $\R$-divisors $B,D\ge0$ on $X$ and nef $\R$-divisors $M',P'$ on $X'$ such that
\begin{enumerate}
\item $(X,B+M)$ and $(X,B+D+M+P)$ are $(n+1)$-dimensional glc pairs with data given by $M'$ and $P'$, where $M=f_*M', P=f_*P'$.
\item $M'=\sum \mu_jM'_j$, where $M'_j$ are relatively nef Cartier divisors and $\mu_j\in I$.
\item $P'=\sum \nu_kP'_k$, where $P'_k$ are relatively nef Cartier divisors and $\nu_k\in I$.
\item The coefficients of $B,D,M'+P'$ belong to $I$.%\footnote{Isn't it enough to say "The coefficients of $B,D,M'+P'$ belong to $I$."?}
\end{enumerate}
We further assume that there
exists a non-gklt center $V$ of $(X,B+D+M+P)$ such that $V$ is not a non-gklt center of $(X,B+M)$ and $\dim V\le\dim X-2$.
Then we can construct a generalized log canonical pair $(S,\Delta+N)$ with $S'\to S$ and $N'$ as in Definition \ref{def: g-pairs} such that
\begin{enumerate}
\item $S$ is a projective variety of
dimension at most $n$,
%\item The coefficients of $N'$ belong to $I$.
\item the coefficients of $\Delta$ belong to $D(I)$,
\item $K_S+\Delta+N$ is numerically trivial,
\item $N'=\sum a_iN'_i$, where $N'_i$ are relatively nef Cartier divisors and $a_i\in I$, and
\end{enumerate}
at least one of the following happens: 
\begin{itemize}
\item[(i)]Some component of $\Delta$ has coefficient of the form
$$
\frac{m-1+\alpha+c}{m}
$$
where $m$ is a positive integer, $\alpha\in I_+$, and $c\in I$ is the coefficient of some component of $D$ or $P'$.
\item[(ii)] $N'_i$ is not numerically trivial for some $i$ and $a_i=g+\nu_k$, where $g\in I$ and $\nu_k>0$ is a coefficient of $P'$.
\end{itemize}
\end{lem}

\begin{proof}
We can replace $V$ with a maximal (with respect to inclusion) glc center of $(X,B+D+M+P)$ satisfying $\dim V\le\dim X-2$ and $V$ is not a glc center of $(X,B+M)$. Let $Q$ be an analytically sufficiently general point of $V$. Consider an
open neighborhood $U$ of $Q$ and a Stein compact subset $W$ of $X$ such that $U\subset W$
and that $\Gamma(W,\OO_X)$ is noetherian. By Theorem \ref{thm: dlt-blowup for g-pairs}, after shrinking $X$ around $W$ suitably, we can construct a projective bimeromorphic morphism $\pi: Y\to X$ with $K_Y+B_Y+D_Y+M_Y+P_Y=\pi^*(K_X+B+D+M+P)$ such that
\begin{enumerate}
    \item $Y$ is $\Q$-factorial over $W$, and $B_Y,D_Y,M_Y,P_Y$ are pushforwards of $B',D',M'$, $P'$ (after possibly replacing $X'$ by a higher model),
    \item $(Y,B_Y+D_Y+M_Y+P_Y)$ is gdlt, where $B_Y+D_Y$ is the boundary part and $M_Y+P_Y$ is the nef part,
    \item $a(E,X,B+D+M+P)=0$ holds for every $\pi$-exceptional divisor $E$, and
    \item there exists a $\pi$-exceptional divisor $F$ such that $\pi(F)=V$.
\end{enumerate}

Let $\hat{D}$ be the birational transform of $D$ on $X'$ and $\hat{D}_Y$ be the pushforward of $\hat{D}$ on $Y$. We first claim that we can choose $F$ in (4) such that $(\hat{D}_Y+P_Y)|_{F_v}$ is not numerically trivial, where $v\in V\cap U$ is an analytically sufficiently general point. %Indeed, as explained in \cite[Section 11]{Fuj22a}, \cite[Lemma 3.6.2]{BCHM10} holds true in the complex analytic setting. 
Let $E=\pi^*(D+P)-\hat{D}_Y-P_Y$, then we can see $E\ge0$ since $(X,B+M)$ is glc and every $\pi$-exceptional divisor $E_i$ has log discrepancy $a(E_i;B+D+M+P)=0$. %\footnote{NOT "by noticing that $E$ is the pushforward of $f^*(D+P)-\hat{D}-P'$, which is effective by the negativity lemma."} 
Moreover, since $V$ is not a log canonical center of $(X,B+M)$, $E$ is non-trivial. Therefore by \cite[Lemma 3.6.2]{BCHM10} (see also \cite[Section 11]{Fuj22a}) there is a component $F$ of $E$ with a covering family of curves $C$ (contracted over $X$) such that $E\cdot C<0$. So $(\hat{D}_Y+P_Y)\cdot C>0$ for such curves and hence $(\hat{D}_Y+P_Y)|_F$ is not numerically trivial over sufficiently general points of $V$.

After replacing $X'$ by a higher model, we may assume that $g:X'\to Y$ is a projective morphism. Let $K_{X'}+\Delta'+M'+P'=f^*(K_X+B+D+M+P)$ and $F'$ be the birational transform of $F$ on $X'$. Let $\Delta_{F'}$ be the $\R$-divisor defined by the adjunction $K_{F'}+\Delta_{F'}=(K_{X'}+\Delta')|_{F'}$ and $\Delta_F,M_F,P_F$ be the pushforwards of $\Delta_{F'},M'|_{F'},P'|_{F'}$, then these data define a generalized pair $(F,\Delta_F+M_F+P_F)$ with nef part $M_F+P_F$ as in Definition \ref{def: g-pairs} and we have
$$
(K_Y+B_Y+D_Y+M_Y+P_Y)|_F\sim_{\R}K_F+\Delta_F+M_F+P_F.
$$
Following \cite[Definition 4.7 and Remark 4.8]{BZ16}, $(F,\Delta_F+M_F+P_F)$ is generalized log canonical. 

Next we calculate the coefficients of $\Delta_F$, cutting by hyperplanes in $Y$ we can assume that $\dim Y=2$ (note that we are working over a Stein set $W$ and $Y$ is projective over $W$). Let $p\in F$ be a point and $l_p$ be the Cartier index at $p\in Y$. In this case we may assume that $F'$ is a normal curve isomorphic to $F$ (since $F$ is already normal), so we may also regard $p$ as a point on $F'$. Then by classification of klt surface singularities we obtain that 
\begin{align*}
    \mult_p\Delta_F&=\frac{l_p-1}{l_p}+\mult_p(B_Y+D_Y-F)|_F+\mult_p((g^*(M_Y+P_Y)-M'-P')|_{F'})\\
    &=\frac{l_p-1+\beta+\gamma}{l_p}\in D(I)
\end{align*}
where $\mult_p(B_Y+D_Y-F)|_F=\frac{\beta}{l_p}$, $\mult_p((g^*(M_Y+P_Y)-M'-P')|_{F'})=\frac{\gamma}{l_p}$ and $\beta,\gamma,\beta+\gamma\in I_+$ by the assumptions on the coefficients of $B+D$ and $M'+P'$. 

Let $S$ (resp. $S'$) be the general fiber of the Stein factorization of $F\to V$ (resp. $F'\to V$), $\Delta=\Delta_F|_S$, $N'=(M'+P')|_{S'}$, $N=(M_F+P_F)|_S$. Then, these data define a generalized pair $(S,\Delta+N)$ with nef part $N$ as in Definition \ref{def: g-pairs} and we have
\begin{enumerate}
    \item $(S,\Delta+N)$ is glc,
    \item $K_S+\Delta+N\sim_{\R}0$,
    \item the coefficients of $N'$ belong to $I$,
    \item $(\hat{D}_Y+P_Y)|_S$ is not numerically trivial.
\end{enumerate}
If $P'|_{S'}$ is not numerically trivial, then (ii) in the statement is satisfied and we are done. So we can assume that $P'|_{S'}\equiv0$, hence $P_F|_S\equiv0$. If we write $g^*(P_Y)=P'+G$ and let $G_S$ be the pushforward of $G|_{S'}$ on $S$, then we have 
$$
(\hat{D}_Y+P_Y)|_S=P_F|_S+G_S+\hat{D}_S\equiv G_S+\hat{D}_S\neq0,
$$
where $\hat{D}_S:=\hat{D}_Y|_S$. Let $R_F$ be the pushforward of $(g^*(M_Y+P_Y)-M'-P')|_{F'}$ 
on $F$ and $R_S:=R_F|_S$, then $\mult_p(R_F)=\frac{\gamma}{l_p}$ in the previous computation. Notice that $\hat{D}_S\le (B_Y+D_Y-F)|_S$ and $G_S\le R_S$, therefore $G_S+\hat{D}_S\neq0$ implies that (i) holds.

%Let $v\in V\cap U$ be an analytically sufficiently general point and $F_v$ be the fiber of $F$ over $v$, which is a normal projective variety. 

%We first claim that $(D_Y+P_Y)|_{F_v}$ is not numerically trivial for some $\pi$-exceptional divisor $F$ that $\pi(F)=V$.

\end{proof}

\begin{thm}\label{thm: t_i is acc}
Let $\Lambda$ be a DCC set of non-negative real numbers and $d$ a positive integer. Assume $X_i,B_i,M_i,M'_i,D_i,P'_i,P'_i$ are as in Definition \ref{def: glct} such that for any $i\ge1$, 
\begin{enumerate}
    \item  $(X_i, B_i+M_i)$ are glc pairs of dimension $d$,
    \item $M_i' = \sum \mu_{i,j} M'_{i,j}$ where $M'_{i,j}$ are relatively nef Cartier divisors and  $\mu_{i,j}\in \Lambda$,
    \item $P_i' = \sum  \nu_{i,k}P'_{i,k}$ where $P'_{i,k}$ are relatively nef Cartier divisors and $\nu_{i,k}\in \Lambda$,
    \item the coefficients of $B_i$ and $D_i$ belong to $\Lambda$,
    \item $(X_i,B_i+M_i+t_iD_i+t_iP_i)$ is glc and has a glc center $V_i$ which is not a glc center of $(X_i,B_i+M_i)$ for some $t_i>0$.
\end{enumerate}
Then $T=\{t_i\}_{i\ge1}$ is an ACC set.
\end{thm}
\begin{proof}
Assume that the sequence $\{t_i\}_{i\ge1}$ is strictly increasing, if $\dim V_i=d-1$, then we have $1-t_i\lambda_i=\lambda_i'\in\Lambda\cap[0,1]$ for some $0<\lambda_i\in\Lambda$. Let 
$$
\Gamma_1:=\{\frac{1}{\lambda}~|~0<\lambda\in\Lambda\},~\Gamma_2:=\{1-\lambda'~|~\lambda'\in\Lambda\cap[0,1]\},
$$ 
then $\Gamma_1$, $\Gamma_2$, and $\Gamma_1\cdot \Gamma_2$ are ACC sets and $\{t_i\}_{i\ge1}\subset\Gamma_1\cdot\Gamma_2$. This contradicts the assumption that $\{t_i\}_{i\ge1}$ is strictly increasing, therefore we can assume that $\dim V_i\le d-2$.

Let $I:=\Lambda\cup(T\cdot\Lambda)\cup (\Lambda+T\cdot\Lambda)$, then $I$ is also a DCC set. %\footnote{Should we consider $D(I^+)$ here?}
Possibly replacing $\{t_i\}_{i\ge1}$ by a subsequence, by Lemma \ref{lem: glct make num trivial g-pair} and \cite[Theorem 1.6]{BZ16}, one of the following happens:
\begin{enumerate}
    \item[(i)] $\frac{m_i-1+\alpha_i+t_ic_i}{m_i}$ belongs to a finite set $\Lambda^0$ for every $i\ge 1$, where $m_i\in\mathbb{N}^*$, $\alpha_i\in I_{+}$ and $0<c_i\in I$.
    \item[(ii)] $g_i+t_iv_{i,k}$ belongs to a finite set $\Lambda^{0}$ for every $i\ge 1$, where $g_i\in I$ and $v_{i,k}>0$.
\end{enumerate}
In either case we can conclude that $t_i$ must belong to a finite set $\Lambda^{1}$, which is a contradiction and we are done.
\end{proof}

It is easy to see that Theorem \ref{thm: glct(I,J) acc} is equivalent to the following theorem by taking $\Lambda=I\cup J$, whose algebraic case is exactly \cite[Theorem 1.5]{BZ16}.

\begin{thm}\label{thm: acc for glct}(ACC for generalized lc thresholds.)
Fix $\Lambda$ be a DCC set of nonnegative real numbers and $d$ a natural
number. Then there is an ACC set $\Theta$ depending only on $\Lambda$ and $d$ such that if
$(X, B + M)$, $M'$, $P'$, $P$, $D$ are as in Definition \ref{def: glct}, and
\begin{enumerate}
\item  $(X, B + M)$ is glc of dimension $d$,
\item $M' = \sum \mu_j M'_j$ where $M'_j$ are relatively nef Cartier divisors and  $\mu _j\in \Lambda$,
\item $P' = \sum  \nu_kP'_k$ where $P'_k$ are relatively nef Cartier divisors and $\nu_k\in \Lambda$, and
\item the coefficients of $B$ and $D$ belong to $\Lambda$,
\end{enumerate}
then the generalized lc threshold ${\rm glct}(X,B + M;D + P)$ of $D + P$ with respect to $(X, B + M)$ belongs
to $\Theta$.

\end{thm}

%Finally we prove the main results. %Theorem \ref{thm: acc for glct}:

\begin{proof}%[Proof of Theorem \ref{thm: acc for glct}]
Suppose that $c={\rm glct}(X,B + M;D + P)>0$, then there exists a non-increasing sequence $c_i\geq c_{i+1}\geq\cdots $ with $\lim_{i\to\infty}c_i=c$ and relatively compact open subsets
$U_i\subset X$ such that $c_i={\rm glct}(U_i,B_i + M_i;D_i + P_i)$ where $B_i + M_i+D_i + P_i=(B + M+D + P)|_{U_i}$. By Remark \ref{rem: lct in relative compact set create lc centers} we know that $(U_i,B_i+M_i+c_iD_i+c_iP_i)$ has a glc center which is not a glc center of $(U_i,B_i+M_i)$. Since the closure of a DCC set is also a DCC set, it suffices to consider relatively compact varieties $X$, and then the statement follows from Theorem \ref{thm: t_i is acc}.
\end{proof}

\begin{proof}[Proof of Theorem \ref{thm: accumulation pts come from lower dim}]
Suppose $c$ is an accumulation point of $\mathrm{GLCT}_n(I)$, then again there exists a non-increasing sequence $c_i\geq c_{i+1}\geq\cdots $ with $\lim_{i\to\infty}c_i=c$ and relatively compact open subsets
$U_i\subset X$ such that $c_i={\rm glct}(U_i,B_i + M_i;D_i + P_i)$ where $B_i + M_i+D_i + P_i=(B + M+D + P)|_{U_i}$. Therefore by Lemma \ref{lem: glct make num trivial g-pair} we know that $c_i\in\mathcal{N}_d(I,\mathbb{N},\mathbb{N})$, which is defined in \cite[Definition 2.18]{Liu18}. Notice that since the generalized pair $(S,\Delta+N)$ constructed in Lemma \ref{lem: glct make num trivial g-pair} is projective,  we are in  the algebraic setting  and so the proof follows from \cite{Liu18}.
\end{proof}

\end{document}